\documentclass{daj}

\usepackage{todonotes}
\usepackage{enumerate}
\usepackage{color}
\usepackage{mathrsfs}
\usepackage{bm}
\usepackage[utf8]{inputenc}
\inputencoding{utf8}
\usepackage{comment}
\usepackage{graphicx}
\usepackage[justification=centering]{caption}
\usepackage{amsthm,amssymb,amsmath}
\usepackage{subfigure}

\dajAUTHORdetails{%
  title = {Semialgebraic Methods and Generalized Sum-Product Phenomena}, 
  author = {Yifan Jing, Souktik Roy, and Chieu-Minh Tran},
  plaintextauthor = {Yifan Jing, Souktik Roy, Chieu-Minh Tran},
  keywords = {expanding polynomials, semialgebraic geometry,  o-minimality},
}   

\dajEDITORdetails{%
   year={2022},
   number={18},
   received={25 December 2020},   
   revised={27 Februrary 2022},    
   published={15 December 2022},  
   doi={10.19086/da.55555},       
}   

\newtheorem{theorem}{Theorem}[section]
\newtheorem{lemma}[theorem]{Lemma}
\newtheorem{proposition}[theorem]{Proposition}
\newtheorem{corollary}[theorem]{Corollary}

\theoremstyle{definition}
\newtheorem{definition}[theorem]{Definition}
\newtheorem{remarkx}[theorem]{Remark}
\newenvironment{remark}
  {\pushQED{\qed}\remarkx}
  {\popQED\endremarkx}

\newtheorem{factx}[theorem]{Fact}
\newenvironment{fact}
  {\pushQED{\qed}\factx}
  {\popQED\endfactx}
\def\RR{\mathbb{R}}
\def\QQ{\mathbb{Q}}
\def\NN{\mathbb{N}}
\def\ZZ{\mathbb{Z}}
\def\CC{\mathbb{C}}
\def\S{\mathcal{S}}
\def\C{\mathscr{C}}
\def\D{\mathscr{D}}
\def\sM{\mathscr{M}}
\DeclareMathOperator{\fchar}{char}
\newcommand{\meno}{}
\newcommand{\ir}{(i)}
\newcommand{\iir}{(ii)}
\newcommand{\iiir}{(iii)}
\newcommand{\ivr}{(iv)}
\newcommand{\vr}{(v)}

\newtheorem{claimx}{Claim}
\newenvironment{claim}
  {\pushQED{\qed}\claimx}
  {\popQED\endclaimx}

\begin{document}

\begin{frontmatter}[classification=text]


\title{Semialgebraic Methods and Generalized Sum-Product Phenomena} 


\author[yifan]{Yifan Jing\thanks{Supported by Ben Green's Simons Investigator Grant, ID:376201.}}
\author[souktik]{Souktik Roy}
\author[minh]{Chieu-Minh Tran\thanks{Supported by Anand Pillay's NSF Grant-2054271.}}


\begin{abstract}
For a bivariate $P(x,y) \in \RR[x,y]\setminus (\RR[x] \cup \RR[y])$, our first result shows that  for all finite $A \subseteq \RR$, $|P(A,A)|\geq \alpha|A|^{5/4}$ with $\alpha =\alpha(\deg P) \in \RR^{>0}$ unless
$$ P(x,y)=f(\gamma u(x)+\delta u(y))  \text{ or }  P(x,y)=f(u^m(x)u^n(y)) $$
for some univariate $f, u \in \RR[t]\setminus \RR$,  constants $\gamma, \delta \in \RR^{\neq 0}$, and $m, n\in  \NN^{\geq 1}$.
This resolves the symmetric nonexpanders classification problem proposed by de Zeeuw. Our second  and third results are sum-product type theorems for two polynomials, generalizing the classical result by Erd\H os and Szemer\'edi as well as a theorem by Shen. We also obtain similar results for $\CC$, and from this deduce results for fields  of characteristic $0$ and fields of large prime characteristic. The proofs of our results use tools from semialgebraic/o-minimal geometry.
\end{abstract}
\end{frontmatter}


\section{Introduction}

  Suppose $K$ is a field, $P(x,y)$ is a bivariate polynomial over $K$, and $P(x,y)$ has nontrivial dependence on $x$ and $y$, in other words, $P(x,y)$ is in  $K[x,y]\setminus(K[x] \cup K[y]$). For $A, B \subseteq K$, we set
$$ P(A,B): = \{ P(a,b) : a \in A, b\in B  \}.$$
We say that $P(x,y)$ is {\bf nonexpanding} over $K$ if for each $\varepsilon \in \RR^{>0}$, there is arbitrarily large $n \in \NN^{\geq 1}$ and   $A,B \subseteq K$ with $|A|=|B|=n$ such that $|P(A,B)|< n^{1+\varepsilon}$. We say that $P(x,y)$ is  {\bf symmetrically nonexpanding} over $K$ if for each $\varepsilon \in \RR^{>0}$, there is arbitrarily large $n \in \NN^{\geq 1}$ and $A\subseteq K$ with $|A|=n$ such that $|P(A,A)|< n^{1+\varepsilon}$.
A symmetric nonexpanding polynomial over $K$ (e.g., $P(x,y)=x+y$, where $A$ can be chosen to be a long arithmetic progression) is clearly nonexpanding over $K$. The converse is not true.
 A well-known result by Elekes and R\'onyai~\cite{Elekes00} provides us with a classification of  nonexpanding polynomials when $K$ is either $\RR$ and $\CC$: The polynomial $P(x,y) $ is nonexpanding  over $K$ if and only if there are nonconstant univariate  polynomials  $f, u, v$ over $K$ such that
\[
\text{either}\quad P(x,y)=f(u(x)+v(y)) \quad\text{or}\quad P(x,y)=f(u(x)v(y)).
\]
With $P(x,y) =x+y^2$, a result by Elekes, Nathanson,
and Ruzsa~\cite{ENR} implies that there is  $\alpha \in \RR^{>0}$ such that for all $A \subseteq \RR$, we have $|P(A,A)|\geq \alpha |A|^{5/4}$. In particular, $x+y^2$ is nonexpanding but not symmetrically nonexpanding. This was pointed out by de Zeeuw  in~\cite{survey}, who then asked for the general form of $P(x,y)$ for similar lower bound on $|P(A,A)|$ to hold.
Our first theorem  answers de Zeeuw's question and gives us a classification of symmetrically nonexpanding polynomials on $\RR$ and $\CC$: 


\begin{theorem}\label{thm: new main}
Suppose $K$ is either $\RR$ or $\CC$, and $P(x,y) \in K[x,y]\setminus(K[x]\cup K[y])$. Then there is $\alpha = \alpha( \deg P)$ such that exactly one of the following possibilities holds:
\begin{enumerate}[{\rm (i)}]
    \item  For all $n \in \NN$ and all $A \subseteq K$ with $|A| =n $, we have $|P(A,A)|  \geq \alpha n^{5/4}. $
   \item $P(x,y)=f(\gamma u(x)+\delta u(y)) $ where  $f$ and $u$ are nonconstant univariate polynomials over $K$, and $\gamma$ and $\delta$ are in $K^{\neq 0}$.
    \item $P(x,y)=f(u^m(x)u^n(y))$ where  $f$ and $u$ are nonconstant  univariate polynomials over $K$, and $m$ and $n$ are in $\NN^{\geq 1}$.
\end{enumerate}
Moreover, $P(x,y)$ is symmetrically nonexpanding over $K$ if an only if it satisfies either (ii) or (iii).
\end{theorem}

 
 
  Erd\H os and Szemer\'edi observed in \cite{ES83} the following sum-product phenomenon: There is $r\in \RR^{>1}$ such that for some $\alpha \in \RR^{>0}$, for all finite  $A \subseteq \ZZ$ 
with $|A|=n$, we have
$$\max\{|A+A|, |AA|\}  \geq \alpha n^r. $$ 
In the same paper, they also conjectured that the exponent $r$ can be taken arbitrarily close to $2$.
Partial results toward this conjecture with explicit $r$ were obtained by Nathanson~\cite{Nathanson} with $r =32/31$, by Ford~\cite{Ford} with $r = 16/15$, by Chen~\cite{Chen} with $r=6/5$, by Elekes~\cite{Elekes97} with $r=5/4$, by Solymosi~\cite{Solymosisoloearlier, Solymosisolo} with $r =14/11-o(1)$ and $r=4/3-o(1)$, and by Konyagin--Shkredov~\cite{KonyaginShkredov, KonyaginShkredov2}, Rudnev--Shkredov--Stevens~\cite{RudnevShkredovStevens}, and  Shakan~\cite{Shakan}, with $r = 4/3+c$ for some explicit small $c \in \RR^{>0}$. Many of the above results are in fact proven for more general contexts, where $\ZZ$ is replaced by either $\RR$ or $\CC$. As polynomials involve both addition and multiplication, there is a general expectation that  ``most'' $P(x,y) \in \RR[x,y]$ already exhibit on their own the sum-product phenomenon $|P(A,A)|\geq \alpha |A|^{r}$ with $\alpha \in \RR^{>0}$ and $r \in \RR^{>1}$. Theorem~\ref{thm: new main} can therefore be seen as a sum-product phenomenon for a single polynomial with $r=5/4$, in which we clarify the meaning of ``most'' $P(x,y) \in \RR[x,y]$.

\meno The proof of some of the above sum-product bounds also yields asymmetric versions for $\max\{|A+B|, |AB|\}$. The asymmetric counterpart of Theorem~\ref{thm: new main} for $\RR$, with a better exponent $r=4/3$, is the following strengthening of the Elekes--R\'onyai theorem by Raz, Sharir, and Solymosi: For $P(x,y) \in \RR[x,y]\setminus(\RR[x] \cup \RR[y])$, there is $\alpha = \alpha(\deg P) \in \RR^{>0}$ such that if there are finite $A,B \subseteq \RR$ with $|A|=|B|=n$ and $P(A,B) < \alpha n^{4/3}$, then  either
$P(x,y)=f(u(x)+v(y))$ or $P(x,y)=f(u(x)v(y))$ for some univariate  polynomials  $f, u, v$ over $\RR$.

\meno Somewhat surprisingly, proving Theorem~\ref{thm: new main} requires a result for two polynomials along the line of the Raz--Sharir--Solymosi theorem. That is our second theorem:
\begin{theorem}\label{thm:main}
Suppose $K$ is $\RR$ or $\CC$,  and $P(x,y)$ and $Q(x,y)$ are bivariate polynomials in $K[x,y]\setminus(K[x]\cup K[y])$. Then there is $\alpha = \alpha( \deg P, \deg Q) $ such that exactly one of the following holds:
\begin{enumerate}[{\rm (i)}]
    \item  For all $n \in \NN$ and subsets $A$ and $B$ of $K$ with $|A| =|B| =n $
        $$ \max\{ |P(A,B)|, |Q(A,B) |   \}  \geq \alpha n^{5/4}. $$
    \item $P(x,y)=f(\gamma_1 u(x)+\delta_1 v(y)) $ and $Q(x,y)=g(\gamma_2 u(x)+\delta_2 v(y))$ where  $f$, $g$, $u$, and $v$ are nonconstant univariate polynomials over $K$ and $\gamma_1$, $\gamma_2$, $\delta_1$, and $\delta_2$ are in $K^{\neq 0}$.
    \item $P(x,y)=f(u^{m_1}(x)v^{n_1}(y)) $ and $Q(x,y)=g(u^{m_2}(x)v^{n_2}(y))$  where  $f$, $g$, $u$, and $v$ are nonconstant univariate polynomials over $K$ and $m_1$, $m_2$, $n_1$, and $n_2$ are in $\NN^{\geq 1}$.
\end{enumerate}
\end{theorem}
   
With $P(x,y)=x+y$ and $Q(x,y)=xy$, it is easy to see that we are not in scenario (ii) or (iii), so we must be in scenario (i). Hence, Theorem~\ref{thm:main} implies the Erd\H os--Szemer\'edi theorem.
An early known special case of Theorem~\ref{thm:main} with $P(x,y)=x+y$ is the main result in the paper ``Algebraic methods in sum-product phenomena'' by Shen~\cite{Shen}. A variant for finite field of Shen's result was obtained by Bukh and Tsimerman \cite{BT12}.


\meno  We also obtain a generalization of Theorems~\ref{thm: new main} (except for the last assertion) for two polynomials. This is also the counterpart of Theorem~\ref{thm:main} for $A=B$: 

 \begin{theorem}\label{thm:sym}
Suppose $K$ is either $\RR$ or $\CC$. With the same  $\alpha = \alpha( \deg P, \deg Q)$ as in Theorem~\ref{thm:main},  exactly one of the following  holds:
\begin{enumerate}[{\rm (i)}]
    \item For all $n \in \NN$ and all $A \subseteq K$ with $|A| =n $
        $$ \max\{ |P(A,A)|, |Q(A,A)|    \}  \geq \alpha n^{5/4}. $$
   \item $P(x,y)=f(\gamma_1 u(x)+\delta_1 u(y)) $ and $Q(x,y)=g(\gamma_2 u(x)+\delta_2 u(y))$  where  $f$, $g$, and $u$ are univariate polynomials over $K$ and $\gamma_1$, $\gamma_2$, $\delta_1$, and $\delta_2$ are in $K^{\neq 0}$.
    \item $P(x,y)=f(u^{m_1}(x)u^{n_1}(y)) $ and $Q(x,y)=g(u^{m_2}(x)u^{n_2}(y))$ where  $f$, $g$, and $u$ are univariate polynomials over $K$ and $m_1$, $m_2$, $n_1$, and $n_2$ are in $\NN^{\geq 1}$.
\end{enumerate}
\end{theorem}

\meno By standard model-theoretic transfer, the analogues of Theorems~\ref{thm: new main}, \ref{thm:main}, and~\ref{thm:sym}  also hold for all algebraically closed fields and real closed fields. With some basic Galois theory, modified versions where ``exactly one of the cases holds'' is replaced by ``at least one of the cases holds'' in the statement of these three theorems can be obtained for all fields of characteristic $0$. Proposition~\ref{prop: transfer} covers both these results. Our results also have some ramifications for fields with large positive characteristic. Results in the same direction for fixed finite fields were obtained by Tao~\cite{Tao} using an algebraic regularity lemma. See Remark~\ref{rem: ending} for the related discussion.


\meno  Erd\H os--Szemer\'edi sum-product type theorems, the Elekes--R\'onyai type theorem, as well as our results can be seen as instances of a more general phenomenon: finite subsets of a one-dimensional space have expansion behavior under a ``sufficiently nice'' binary function unless the situation is ``controlled'' by a single abelian group. One reasonable interpretation of ``sufficiently nice'' is to be definable in a tame structure in the sense of model theory.
The current paper and recent qualitative generalizations of the Elekes--R\'onyai theorem by Bays, Breuillard, Chernikov, Starchenko, and Peterzil~\cite{MartinBays, ArtemSergeiESomin, ArtemSergeiESstronglymin} can be seen as two pieces of evidence for this view. (These other results are stated in the framework of points on a grid satisfying a ``sufficiently nice'' relation; see~\cite[Section~2]{survey} for a translation between the frameworks. A version of Theorem~\ref{thm:main} without an explicit exponent in the bound and in a less precise form can in fact be deduced from the result by Bays and Breuillard~\cite{MartinBays}; see Remark~\ref{Rem: Second comment}.) 
Also using tools from model theory are the works of  Hrushovski \cite{Hrushovski} and then of Breuillard, Green, and Tao \cite{BGT} on the structure of approximate groups; the  classification in the latter is not an instance of the above phenomenon but it is very close in spirit; see \cite{Hrushovski2} for a related survey.

\meno We now describe in more detail the strategies of the proofs and the structure of the paper.  After the preliminary Section~2 on logic and o-minimality and Section~3 on standard nonexpanding examples, the proof of Theorem~\ref{thm:main} is given in Section~4. This has three steps. First, we apply the strengthening \cite{RSS16} by Raz--Sharir--Solymosi of Elekes--R\'onyai's result in \cite{Elekes00}  and reduce to the case where $P(x,y)$ is of the form $f_1(u_1(x)+v_1(y))$  or $f_1(u_1(x)v_1(y))$, and $Q(x,y)$ is of the form $f_2(u_2(x)+v_2(y))$  or $f_2(u_2(x)v_2(y))$,
where $f_1$, $f_2$, $u_1$, $u_2$, $v_1$, and $v_2$ are univariate polynomials over $K$.
In the second step, we use a suitable generalization of the Szemer\'edi--Trotter Theorem~\cite{SzT} (the result for $\RR$ comes from \cite{Fox17} by Fox, Pach, Sheffer, Suk, and Zahl, and that for $\CC$ can be deduced from \cite{SolymosideZeeuw} by Solymosi and de Zeeuw)
in the same fashion as in Elekes' classical proof of sum-product phenomena ~\cite{Elekes97}. This allows us to prove that (i) happens except for a special situation. In the final third step, we  analyze the above special situation to show that either (ii) or (iii) must happen. This last step employs ideas from semialgebraic geometry/o-minimal geometry, in particular, a definable Ramsey-type result. It is quite close in principle to the proof in \cite{Shen} and the proof in \cite{RSS16}, but we use  o-minimality instead of algebraic geometry.

\meno The third step is readily generalizable to many other settings related to o-minimality. We expect that Theorem~\ref{thm:main} admits suitable generalizations as well, once the corresponding ingredients for the first and second steps are developed; in some cases, they already are. See Remark~\ref{Rem: Generalization remark} for a more detailed discussion.


\meno Finally, the proof of Theorems~\ref{thm: new main} and~\ref{thm:sym} are given in Section~5. The key idea is to apply Theorem~\ref{thm:main} for the two polynomials $P(x,y)$ and $\hat{P}(x,y)$ where we set $\hat{P}(x,y):= P(y,x)$. 
We also discuss here generalizations of the main theorems to fields other than $\RR$ and $\CC$ in this section.

\subsection*{Notation and conventions} In the rest of the paper, $m$ and $n$ range over the set $\NN =\{0, 1, \ldots \}$ of natural numbers, $k$ and $l$ range over the set $\ZZ$ of integers, $K$ is a field, and $P(x,y)$ and $Q(x,y)$ are bivariate polynomials over $K$ with has nontrivial dependence on $x$ and $y$ (i.e., $P(x,y)$ and $Q(x,y)$ are in  $K[x,y]\setminus(K[x] \cup K[y]$)),  $\RR$ and $\CC$ are the fields of real and complex numbers, and $(\RR, <)$ is the ordered field of real numbers. For $f, g: \NN \to \RR^{> 0}$, we write $f(n)=O(g(n))$ if there is some $M\ \in \RR^{>0}$ such that $|f(n)|<Mg(n)$ for all $n$. We write $f(n)\asymp g(n)$ if $f(n)=O(g(n))$ and $g(n)=O(f(n))$.

\section{Preliminaries on Logic and o-minimality}

To make the paper more accessible to readers without a background in model theory, we include here a brief introduction to the topics of logic, model theory, and semialgebraic geometry (through the lens of o-minimality).
We will keep the discussion brief, focusing on motivations and the general picture. For more systematic accounts of model theory in general and o-minimality in particular see \cite{Marker, TentZiegler} and~\cite{Lou} respectively.

\meno  The logician's notion of a {\it structure} is a generalization of the notion of a field. In this paper, the main structures are $\RR$, $\CC$, and $(\RR, <)$. The method of this paper, modulo some missing ingredients, looks ready for generalization to some {\it expansions} of $\RR$ and $\CC$ (i.e.,  structures enriching $\RR$ and $\CC$); examples include the exponential field $(\RR, \text{exp})$ of real numbers and the expansion $\CC_{\text{an}}$ of $\CC$ obtained by adding restrictions of analytic functions to bounded closed disks. So the reader may want to keep such structures in mind as well; see Remark~\ref{Rem: Generalization remark} for further details.

\meno  The notion of {\it definable set} (always assumed to be with parameters) is the corresponding generalization of the notion of {\it algebraic set} (i.e., solution set of a system of polynomial equations in a field).
In a field $K$, definable sets include algebraic sets and also sets that can be obtained from these through finitely many applications of taking intersections, unions, complements, and projections to lower dimensions. For instance, the set 
$$X = \{ (a, b) \in K^2 : \text{ there are } c, d\in K \text{ with } c^2+d^2 =a  \text{ and } c^3+d^3 =b\}$$  is definable in $K$ as it is the projection of the algebraic subset of $K^4$ defined by the system $x = z^2+t^2, y =z^3+t^3$ onto the first two coordinates. By the above description, $K^2 \setminus X$ and the projection of $K^2 \setminus X$ onto the first coordinate are also definable. Definable sets are ``solution sets'' of first-order formulas (which, in the case of fields, involve logical symbols like $\wedge, \vee, \neg, \exists, \forall$ on top of field-theoretic operations). For example, $X$ is the ``solution set'' of the first-order formula
$$\exists z\exists t (x = z^2 +t^2 \wedge y= z^3+t^3).$$
Definable sets in an ordered field $(K, <)$ can described similarly, but with the role of algebraic sets replaced by that of {\it semialgebraic sets} (i.e., solution sets of systems of polynomial inequalities). In fact, definability in $(\RR, <)$ and in $\RR$ are equivalent because  for $a$ and $b$ in $\RR$, we have  $a<b$ if and only if there is $c \in \RR \setminus\{0\}$ such that $a+c^2 =b$. For later use, a first-order formula is a first-order {\it statement} if all of its variables are attached to a quantifier $\exists$ or $\forall$.

\meno The idea of definability extends to more complicated objects like functions, families of sets, groups, etc. In a structure with underlying set $K$, a function $f: X \to K^n$ with $X \subseteq K^m$ is {\bf definable}  if the graph of $f$ is definable; in particular, this implies that $X$ and $\text{Image}(f)$ are definable. Still in the same structure, a family $(X_b)_{b\in Y}$ of subsets of $K^m$ is {\bf definable} if $Y$ is a definable subset of $K^n$ for some $n$, and the set $ X = \{(a,b) \in K^{m+n} : a \in X_b, b \in Y \}$ is definable.

\meno 
Compared to more restricted notions like algebraic sets, definable sets are very versatile. For example, if $f: X \to \RR^n$ is definable in $(\RR, <)$, then the set
$$\{ a \in X : f \text{ is differentiable with continuous derivatives at } a   \}$$
is definable simply because all the relevant concepts admit epsilon-delta definitions. Unfortunately, this flexibility often comes with the burden that definable sets are often overly complicated and resist geometric understanding.  Model theory is, to a certain extent, the study of structures where we do not have these shortcomings. Examples of such structures are $\RR$, $\CC$, and $(\RR, <)$; this is reflected by Fact~\ref{fact: 2} below. Fact~\ref{fact: 2}(i) implies that definable sets in $\RR$ are essentially well-behaved manifolds, which conceptually explain Fact~\ref{fact: 2}(ii)  below.

\begin{fact} \label{fact: 2}
Let $X$ be definable in $\RR$, equivalently, definable in $(\RR, <)$, and let $X'$ be definable in $\CC$. Then we have the following:
\begin{enumerate}
     \item[\ir] (Tarski--Seidenberg Theorem) $X$ is semialgebraic \cite[Theorem 3.3.15]{Marker}.
     \item[\iir] (Abstract cell decomposition) $X$ is a disjoint finite union $\sqcup_{i \in I} X_i$, where $X_i$ is an $\RR$-definably homeomorphic image of $(0,1)^{k_i}$ with $k_i \in \{0, \ldots, m\}$;      hence, by the Invariance of Domain Theorem,   $X_i$ is a connected open set when $k_i =m$ and an open interval when $k_i=m=1$ for each $i \in I$;  moreover, if $f: X \to \RR$ is a definable function, we can choose  $(X_i)_{i \in I}$ such that $f$ is continuous on $X_i$ for each $i \in I$ \cite[Theorem 3.2.11]{Lou}.
     \item[\iiir] (Chevalley--Tarski Theorem) $X'$ is constructible (i.e., a boolean combination of algebraic sets in $\CC$, or equivalently, a finite union of quasi-affine varieties over $\CC$) \cite[Theorem 3.2.2]{Marker}. \qedhere
\end{enumerate}
\end{fact}

  Another advantage in dealing with definable sets is a built-in mechanism for induction, as the image of a definable set under projection to fewer coordinates remains definable. This is particularly powerful when we have good notions of dimension.
By Fact~\ref{fact: 2}(i) and Fact~\ref{fact: 2}(iii), definable sets in $\RR$ are essentially real manifolds, and definable sets in $\CC$ are essentially algebraic varieties. Hence, it is not too surprising that there are good notions of dimension $\dim_\RR$ and $\dim_\CC$ in these cases. Fact~\ref{fact: 3} is about the properties of $\dim_\RR$ and $\dim_\CC$. The properties of $\dim_\RR$ are generally consequences of the fact that $\dim_\RR$ agrees with the usual of notion of dimension of a real manifold ~\cite[page~5]{Lou}, and real closed fields are o-minimal~\cite[Corollary 2.3.6]{Lou}. The properties of $\dim_\CC$ are generally consequences of the fact that $\dim_\CC$ agrees with dimension in the sense of algebraic geometry~\cite[Corollary 6.2.23]{Marker}, and algebraically closed fields are strongly minimal~\cite[Corollary 3.2.9]{Marker}. 

\begin{fact} \label{fact: 3}

Suppose $K$ is either $\RR$ or $\CC$, and  $X$ and $Y$ are definable in $K$. Then we have the following:
\begin{enumerate}
    \item[\ir] $\dim_K (\emptyset) = -\infty $; $\dim_K (X) =0$  if and only if $X$ is finite; $\dim_\RR (I)  =1$ with $I=(a,b)$, $a,b \in \RR\cup\{-\infty, +\infty\}$ and $a<b$; $\dim_\CC (\CC) =1$. \cite[page~5]{Lou}\cite[Corollary 6.2.23]{Marker}
    \item[\iir] $\dim_K(X \cup Y) =\max\{ \dim_K (X), \dim_K (Y)\}$ \cite[Proposition 4.1.3]{Lou} \cite[Lemma 6.2.7]{Marker}.
    \item[\iiir] If $(X_b)_{b\in Y}$ is a definable family, and $X = \{(a,b) :a \in X_b\}$, then for each $d \leq \dim (X)$, we have $ Y_d = \{ b\in Y : \dim_K(X_b) =d \} $    is definable, 
    and 
    $$ \dim_K (X) = \max_{d \leq \dim (X)} (d + \dim_K (Y_d));       $$
    in particular, if $f: X \to Y$ is a definable bijection, then $\dim_K(X) = \dim_K (Y)$, and  $\dim_K (X \times Y ) = \dim_K (X) + \dim_K (Y)$ \cite[Corollary 4.1.6]{Lou} \cite[Lemma 6.2.20, Theorem 6.6.37]{Marker}.
    \item[\ivr]  If $K=\RR$ and $X \subseteq \RR^m$, then $\dim_\RR (X) = m$ if and only if $X$ contains an open subset \cite[Section 4.1.1]{Lou}. 
    \item[\vr] (algebraic boundedness) If $(X_b)_{b\in Y}$ is a definable family in $K$, then there is $N \in \NN^{\geq 1}$ such that either $|X_b|< N$ or $\dim_K(X_b) \geq 1$ \cite[Lemma 3.1.7, Corollary 4.1.6]{Lou} \cite[Lemma 6.1.14, Lemma 6.2.20]{Marker}. \qedhere
\end{enumerate}
\end{fact}

  An expansion $( \RR, \ldots)$ of $\RR$ is {\bf o-minimal} if every $X \subseteq \RR$ definable in  $(\RR, \ldots)$ is a finite union of points and open intervals. The analogues of Fact~\ref{fact: 2}(ii) and Fact~\ref{fact: 3} hold in all such  $( \RR, \ldots)$ \cite{Lou}. In fact, for an expansion of the field $\RR$, one can take Fact~\ref{fact: 2}(ii) as an alternative definition for o-minimality; indeed, homeomorphic images of $(0,1)^0$ are points and homeomorphic images of $(0,1)^1$ are open intervals. An important example of an o-minimal expansion of $\RR$ is $( \RR, \text{exp})$ where $\text{exp}$ is the exponential map; this allows us to include transcendental functions while keeping algebraic features (e.g., an analogue of Fact~\ref{fact: 3}(v)).

\section{Standard nonexpanding examples }



  Let $K$ be either $\RR$ or $\CC$.  We  will establish that polynomials of forms in (ii) and (iii) of Theorem~\ref{thm: new main} are symmetrically nonexpanding over $K$. We will also prove several other results along the same line.

\meno Let  $f$, $g$, $u$, and $v$ range over nonconstant univariate polynomials over $K$. 
When
$$ P_{\mathrm{a}}(x,y)=  f(u(x)+v(y))\quad \text{and} \quad P_{\mathrm{m}}(x,y)=  f(u(x)v(y)), $$
we say that  $ P_{\mathrm{a}}(x,y)$ is {\bf additive} over $K$, and  $P_{\mathrm{m}}(x,y)$ is {\bf multiplicative} over $K$.  For $\gamma, \delta \in K^{\neq 0}$,  $m, n \in \NN^{\geq 1}$, and 
$$ P_{\mathrm{sa}}(x,y) = f(\gamma u(x)+\delta u(y))  \quad \text{and} \quad   P_{\mathrm{sm}}(x,y)= f(u^m(x)u^n(y)),$$
we say that $ P_{\mathrm{sa}}(x,y)$ is {\bf symmetrically additive} over $K$, and  $ P_{\mathrm{sm}}(x,y)$ is {\bf symmetrically multiplicative} over $K$. For $\gamma_1, \gamma_2, \delta_1, \delta_2 \in K^{\neq 0}$, $m_1, m_2, n_1, n_2 \in \NN^{\geq 1}$, and 
$$ P_{\mathrm{a}}(x,y) = f(\gamma_1 u(x)+ \delta_1 v(y)), \quad  Q_{\mathrm{a}}(x,y) = g(\gamma_2 u(x)+ \delta_2 v(y)),   $$
$$  P_{\mathrm{m}}(x,y) = f(u^{m_1}(x)v^{n_1}(y))), \quad Q_{\mathrm{m}}(x,y) = g(u^{m_2}(x)v^{n_2}(y)),   $$
we say that $P_{\mathrm{a}}(x,y)$ and $Q_{\mathrm{a}}(x,y)$ form an {\bf additive pair} over $K$, and $ P_{\mathrm{m}}(x,y)$ and  $Q_{\mathrm{m}}(x,y)$ form a {\bf multiplicative pair} over $K$.
For $\gamma_1, \gamma_2, \delta_1, \delta_2 \in K^{\neq 0}$, $m_1, m_2, n_1, n_2 \in \NN^{\geq 1}$, and 
$$ P_{\mathrm{sa}}(x,y) = f(\gamma_1 u(x)+ \delta_1 u(y)), \quad  Q_{\mathrm{sa}}(x,y) = g(\gamma_2 u(x)+ \delta_2 u(y)),   $$
$$  P_{\mathrm{sm}}(x,y) = f(u^{m_1}(x)u^{n_1}(y))), \quad  Q_{\mathrm{sm}}(x,y) = g(u^{m_2}(x)u^{n_2}(y)),   $$
we say that $P_{\mathrm{sa}}(x,y)$ and $ Q_{\mathrm{sa}}(x,y)$ form a {\bf symmetric additive pair} over $K$, and $P_{\mathrm{sm}}(x,y)$ and  $Q_{\mathrm{sm}}(x,y)$ form a {\bf symmetric multiplicative pair} over $K$.

\meno In the following lemma,  if $L$ is a field extending $K$, we denote $\mathrm{trdeg}(L/K)$ the trancendental degree of $L$ over $K$.

\begin{lemma} \label{lem: estimation}
Let $\tau_1, \tau_2$ be in $\CC$. For positive $k, l$, we considered the generalized arithmetic progression
$$ G(k, l, \tau_1, \tau_2):= \left\{  \sum_{0 \leq i,j <l} a_{ij}  \tau_1^i \tau_2^j : a_{ij}\in \ZZ, 0 \leq a_{ij} < k \right\}.$$
Then we have the following:
\begin{itemize}
    \item[(i)] If $\mathrm{trdeg}(\QQ(\tau_1, \tau_2)/\QQ)=0$, then 
     there is $h= h(\tau_1, \tau_2) \in \NN^{\geq 1}$ such that for sufficiently large $l$, we have $|G(k,l, \tau_1, \tau_2)|\asymp_{l, \tau_1, \tau_2}k^h$.
    \item[(ii)] If $\mathrm{trdeg}(\QQ(\tau_1, \tau_2)/\QQ)=1$, 
    then there is $h= h(\tau_1, \tau_2) \in \NN^{\geq 1}$ such that for sufficiently large $l$, $|G(k,l, \tau_1, \tau_2)|\asymp_{l, \tau_1, \tau_2}k^{lh}$.
    \item[(iii)] If $\mathrm{trdeg}(\QQ(\tau_1, \tau_2)/\QQ)=2$, then 
     $|G(k,l, \tau_1, \tau_2)|= k^{l^2}$.
    \end{itemize}
\end{lemma}

\begin{proof}
We will only prove statement (i) as statement (ii) can be proven in a similar fashion, and statement (iii) is immediate. Let $d_1$ be the degree of the minimal polynomial of $\tau_1$ over $\QQ$ and $d_2$ be the degree of the minimal polynomial of $\tau_2$ over $\QQ(\tau_1)$. It suffices to show 
$$|G(k,l, \tau_1, \tau_2)|\asymp_{l, \tau_1, \tau_2}k^h \text{ with } h=d_1d_2 \text{ and } l\geq \max\{d_1-1, d_2-1\}.$$ Note that $(\tau_1^i\tau_2^j)_{0\leq i <d_1, 0\leq j<  d_2 }$ are linearly independent. Hence, 
$|G(k,l, \tau_1, \tau_2)|\geq k^h$. On the other hand, by using the minimal polynomial of $\tau_1$ over $\QQ$ and the minimal polynomial of $\tau_2$ over $\QQ(\tau_1)$, each element of $G(k,l, \tau_1, \tau_2)$ is equal to  a sum
$$ \sum_{0 \leq i< d_1, 0\leq j <d_2} \frac{b_{ij}}{N}  \tau_1^i \tau_2^j  $$
where $N=N(l, \tau_1, \tau_2)$ is a positive integer, and $b_{i,j}$  are integers with $|b_{i,j}| \leq Mk $ for some positive integers $M=M(l, \tau_1, \tau_2)$. Hence, $|G(k,l, \tau_1, \tau_2)|= O_{l, \tau_1, \tau_2}(k^h)$ and the desired conclusion follows.
\end{proof}

\begin{proposition} \label{prop: additivemultiplicativenonexpanding}
Suppose $K$ is either $\RR$ or $\CC$. Then we have the following:
\begin{enumerate}[{\rm (i)}]
    \item If $P(x,y)$ is either additive or multiplicative, then for each $\varepsilon \in \RR^{>0}$, there are  $A, B \subseteq K$ with arbitrarily large $n=|A|=|B|$ such that 
    $$  |P(A,B)|<  n^{1+\varepsilon}.$$
    In other words, $P(x,y)$ is nonexpanding.
    \item If $P(x,y)$ is either symmetrically additive or symmetrically multiplicative, then for each $\varepsilon \in\RR^{>0}$, there is $A\subseteq K$ with arbitrarily large $n=|A|$ such that 
    $$  |P(A,A)|< n^{1+\varepsilon}.$$
    In other words, $P(x,y)$ is symmetrically nonexpanding.
    \item If $P(x,y)$ and $Q(x,y)$ forms an additive pair or a multiplicative pair,  then for each $\varepsilon \in \RR^{>0}$, there are  $A, B \subseteq K$ with arbitrarily large $n=|A|=|B|$ such that 
    $$ \max\{ |P(A,B)|, |Q(A,B)| \}< n^{1+\varepsilon}.$$
    \item If $P(x,y)$ and $Q(x,y)$ forms a symmetric additive pair or a symmetric multiplicative pair, then for each $\varepsilon \in\RR^{>0}$, there is $A\subseteq K$ with arbitrarily large $n=|A|$ such that 
    $$ \max\{ |P(A,A)|, |Q(A,A)| \}< n^{1+\varepsilon}.$$
\end{enumerate}
\end{proposition}

\begin{proof}
We note that (i) follows from (iii), and (ii) follows from (iv), so we only need to prove (iii) and (iv). We will only discuss the situation where $P(x,y)$ and $Q(x,y)$ form an additive pairs in (iii) as the remaining parts of (iii) and (iv) can be handled in a similar fashion. Consider first the special case where 
$$P(x,y)=x+\tau_1y \text{ and } Q(x,y)=x+\tau_2y$$
with  $\tau_1, \tau_2 \in \RR^{\neq 0}$. Define $G(k,l, \tau_1, \tau_2)$ as in Lemma~\ref{lem: estimation}.  For $i\in \{1, 2\}$, we have
$$G(k,l, \tau_1, \tau_2)+ \tau_i G(k,l, \tau_1, \tau_2) \subseteq  G(2k,l+1, \tau_1, \tau_2).$$ 
By Lemma~\ref{lem: estimation}, we can choose  $l =l(\tau_1, \tau_2)$ such that for sufficiently large $k$, we have
$ |G(2k,l+1, \tau_1, \tau_2)| < |G(k,l, \tau_1, \tau_2)|^{1+\varepsilon}.$ Hence, choosing $A = B= G(k,l, \tau_1, \tau_2)$ with such $k$ and $l$, we see that $|A+\tau_1B|$ and $|A+\tau_2B|$ has size bounded above by $n^{1+\varepsilon}$ with $n=|A|=|B|$. 

In the general case, 
$P(x,y)=f(\gamma_1u(x)+\delta_1v(y)) \text{ and } Q(x,y)=g(\gamma_2u(x)+\delta_2v(y)).$
Setting $\tau_1=\delta_1/\gamma_1 $ and $\tau_2=\delta_2/\gamma_2$ and replacing $f(t)$ with $f(\gamma_1t)$ and $g(t)$ with $g(\gamma_2t)$ if necessary, we can reduce the problem to the case where
$$P(x,y)=f(u(x)+\tau_1v(y)) \text{ and } Q(x,y)=g(u(x)+\tau_2v(y)).$$
Let $G(k,l, \tau_1, \tau_2)$ be as before, and set $A = u^{-1} G(k, l, \tau_1, \tau_2)$ and $B = v^{-1} G(k, l, \tau_1, \tau_2)$. Argue as in the earlier special case, we arrive at the desired conclusion.
\end{proof}

\section{Proof of Theorem~\ref{thm:main}}\label{sec:real}

  In this section, we will prove Theorem~\ref{thm:main}. For convenience, we will use the terms $\RR$-definable and $\CC$-definable as short-hands for definable in $\RR$ and definable in $\CC$ respectively.

\meno Suppose $r$ is in $(1,2]_\RR$ and $\alpha$ is in $\RR^{>0}$. We say that $P(x, y)$ is {\bf $(r, \alpha)$-expanding} over $K$ if for all $n\in\NN^{>0}$ and all finite $A, B \subseteq K$ with $|A| =|B| =n $, we have
 $$ |P(A,B)|  \geq  \alpha n^{r}. $$
 If for all finite $A\subseteq K$ with $|A|=n $ we have $|P(A,A)|  \geq  \alpha n^{r}$, we say that $P(x, y)$  a {\bf symmetrically $(r, \alpha)$-expanding} over $K$. In particular, a nonexpanding polynomial is not $(r, \alpha)$-expanding, and a symmetric nonexpanding polynomial is not $(r, \alpha)$-expanding.
 We say that $P(x,y)$ and $Q(x, y)$ form a {\bf $(r, \alpha)$-expanding pair} over $K$ if  for all $n\in\NN^{>0}$ and all finite  $A, B\subseteq K$ with $|A|=|B|=n $, we have
 $$\max\{|P(A,B)|, |Q(A,B)|\} \geq  \alpha n^{r}.$$ Finally, if $\max\{|P(A,A)|, |Q(A,A)|\} \geq  \alpha n^{r}$ for all finite $A\subseteq K$ with $|A|=n $, we say that $P(x,y)$ and $Q(x, y)$ form a {\bf symmetric $(r, \alpha)$-expanding pair} over $K$. Clearly, if $(P(x,y),P(x,y))$ is an  $(r, \alpha)$-expanding pair over $K$, then  $P(x,y)$ is $(r, \alpha)$-expanding over $K$.

\meno We will reduce Theorem~\ref{thm:main} to a special case using Fact~\ref{fact:structure}, which consists of Elekes--R\'onyai type structural results for nonexpanding polynomials. The case where $K = \RR$ of Fact~\ref{fact:structure} is a  recent Theorem  by Raz, Sharir, and Solymosi~\cite{RSS16}; this refines Elekes--R\'onyai's original result in \cite{Elekes00}. The case where $K=\CC$ is folklore as communicated privately to us by de Zeeuw. This follows from the arguments used in the proof in \cite{RSS16} together with  Solymosi--de Zeeuw's incidence bound in~\cite{SolymosideZeeuw}. We remark that a Szemer\'edi--Trotter-type lemma proved in \cite{RSS16} used Sz\'ekely's crossing lemma strategy~\cite{Szekely}, this can be replaced by the partitioning method used in the proofs of complex Szemer\'edi--Trotter-type results in the complex field case, see~\cite{ST12, Zahl}.

\begin{fact}\label{fact:structure}
Suppose $K$ is either $\RR$ or $\CC$. Then there is $\alpha = \alpha(\deg P)$ such that if $P(x,y)\in K[x,y]$ is not $(4/3, \alpha)$-expanding over $K$, then $P(x,y)$ is either additive or multiplicative.
\end{fact}

Using Fact~\ref{fact:structure}, the following corollary is immediate.

\begin{corollary} \label{cor:structure}
If Theorem~\ref{thm:main} holds in the special case where $P(x,y)$ and $Q(x,y)$ are individually either additive or multiplicative, then Theorem~\ref{thm:main} holds in general.
\end{corollary}


\meno For $(b_1, b_2) \in K^2$, denote by $C_{b_1, b_2}$ the curve
$$
\big\{\big(P(a,b_1),Q(a, b_2)\big): a\in K\big\}.
$$
Note that $\dim_K(C_{b_1, b_2})\leq 1$ for each $(b_1, b_2) \in K^2$ by Fact~\ref{fact: 3}(iii).
For a  $K$-definable $ S \subseteq K$,  set 
$\C_{P,Q}(S)$ to be the definable family $(C_{b_1,b_2})_{(b_1, b_2) \in S^2}$.
We say that $\C_{P,Q}(S)$ is {\bf scattered} if for all $(b_1,b_2) \in S^2$, 
$$ \{ (b'_1, b'_2) \in S^2 :   \dim_{K}(C_{b_1,b_2} \cap C_{b'_1,b'_2}) = 1 \} \, \text{ is finite} .$$
 We will see that the feature that distinguishes case (i) from cases (ii) and (iii) of Theorem~\ref{thm:main} when $P(x,y)$ and $Q(x,y)$ are each either additive or multiplicative is essentially the existence of a cofinite subset $S$ of $K$ such that  $\C_{P,Q}(S)$ is scattered.

\meno We need a suitable version of  Szemer\'edi--Trotter theorem. The case where $K= \RR$ of Fact~\ref{fact:st} is a very special case of the main result in~\cite{Fox17}. The case where $K = \CC$ is essentially known in the field as privately communicated to us by de Zeeuw. This can be deduced from a result in~\cite{SolymosideZeeuw}. To be more precise, Fact~\ref{fact:st} with the $K_{k,k}$-free condition replaced by a $K_{2,k}$-free condition can be derived quite immediately. To get the actual version of Fact~\ref{fact:st}, we can employ the trick in the proof of Corollary 4.2 of \cite{MPVZ} by Mojarrad, Pham, Valculescu, and de Zeeuw. 
\begin{fact}\label{fact:st}
Assume  $K$ is either $\RR$ or $\CC$, $\phi(x_1, x_2, y_1, y_2)$ is a formula in the language of fields (possibly with parameters from $K$), and  $G = ( X, Y, E)$ is a bipartite graph with the following properties:
\begin{enumerate}
    \item[(i)] $X = X_1 \times X_2$ where $X_1$ and $X_2$ are finite subsets of $K$;
    \item[(ii)] $Y$ is a finite subset of $K^2$;
    \item[(iii)] for $c \in X $ and $d \in Y$, $(c,d)$ is in $E$ if and only if $\phi(c, d)$ holds.
\end{enumerate}{}
If $G$ is $K_{k,k}$-free, then
\[
|E|\leq \beta\big((|X||Y|)^{2/3}+|X|+|Y|\big),
\]
where $\beta$ depends only on $k$ and the complexity of $\phi$ as described in \cite{Fox17}. More precisely, as $\CC$ and $(\RR, <)$ admit quantifier elimination (i.e, Fact~\ref{fact: 2} (i) and (iii)), the formula $\phi$ is equivalent to either a boolean combination of polynomial equations or a boolean combination of polynomial inequalities. The complexity of $\phi$ consists of the length of the Boolean combination and the maximum degree of the polynomials involved.
\end{fact}
  The  generalization of Fact~\ref{fact:st} where $X$ is just assumed to be finite subset of $K^2$ but not necessarily a Cartesian product as in (i) also recently became available. For $\RR$, this is still a special case of the main theorem of~\cite{Fox17}. The result for $\CC$ can be obtained from a recent paper of Walsh~\cite{Miguel} by using the same trick from~\cite{MPVZ} by Mojarrad, Pham, Valculescu, and de Zeeuw  as discussed above.

\begin{corollary} \label{cor:st}  Assume $K$ is either $\RR$ or $\CC$, $N$ is in $\NN^{>0}$, and there is a cofinite $S \subseteq K$ such that $|K\setminus S| \leq N$ and $\C_{P,Q}(S)$ is scattered. 
Then $P(x,y)$ and $Q(x,y)$  form a $(5/4, \alpha)$-expanding pair with $\alpha=  \alpha( \deg P, \deg Q, N)$.
\end{corollary}

\begin{proof}
Recall that  $\C_{P,Q}(S)$ is the $K$-definable family $(C_{b_1,b_2})_{(b_1, b_2) \in S^2}$. 
The family $$(C_{b_1,b_2} \cap  C_{b'_1,b'_2})_{(b_1,b_2, b'_1, b'_2) \in S^4}$$ is also $K$-definable.
For $(b_1,b_2) \in S^2$, set 
$$
Y_{b_1,b_2} = \{ (b'_1, b'_2) \in S^2 :   \dim_K (C_{b_1,b_2} \cap  C_{b'_1, b'_2}) = 1 \}.
$$  
Then the family $(Y_{b_1,b_2})_{ (b_1,b_2) \in S^2}$ is $K$-definable by Fact~\ref{fact: 3}(iii). Using  Fact~\ref{fact: 3}(v), we obtain $k=\max\{N(C_{b_1,b_2} \cap  C_{b'_1,b'_2}),N(Y_{b_1,b_2})\}>0$, where $N(C_{b_1,b_2} \cap  C_{b'_1,b'_2})$ and $N(Y_{b_1,b_2})$ are constants obtained from Fact~\ref{fact: 3}(v), such that for all $(b_1,b_2) \in S^2$, the following holds:
\begin{enumerate}
    \item For all $(b'_1,b'_2) \in K^2$, if $|C_{b_1,b_2} \cap C_{b'_1, b'_2}| \geq k$, then $\dim_K(C_{b_1,b_2} \cap C_{b'_1,b'_2}) = 1$.
    \item If $|Y_{b_1, b_2}| \geq  k$, then $\dim_K (Y_{b_1, b_2}) \geq  1.$
\end{enumerate}
Let  $\phi(x_1, x_2, y_1, y_2)$ be the formula in the language of fields such that $\phi(c_1,c_2, d_1, d_2)$ holds for $(c_1, c_2) \in K^2$ and $(d_1, d_2) \in S^2$ if and only if 
$$ (c_1,c_2)\in C_{d_1, d_2}. $$
For subsets $A$ and $B$ of $K$ with $|A| =|B| =n$, define $G=(X,Y, E)$ as in Fact~\ref{fact:st} with $X_1 = P(A,B\cap S)$, $X_2 = Q(A, B\cap S)$, $Y=  (B\cap S) \times (B\cap S)$, and $\phi(x_1,x_2, y_1, y_2)$ as above.
By the assumption that $\C_{P,Q}(S)$ is scattered and our choice of $k$, the graph  $G$ is $K_{k,k}$-free. Applying Fact~\ref{fact:st}, and noting that $|Y|=|(B\cap S) \times (B\cap S)|\leq n^2$, we get a constant $\beta$ depending only on $\phi$ and $k$ such that
\[
|E| \leq \beta\big(|P(A,B\cap S)\times Q(A,B\cap S)|^{\frac{2}{3}}(n^2)^{\frac{2}{3}} + |P(A,B\cap S)\times Q(A,B\cap S)|+n^2\big).
\]
Let $d$ be be maximum degree of $P(x,y)$ and $Q(x,y)$.
Note that the curve $C_{b_1, b_2}$ passes through at least $n\slash d$ points in $X$ for each $(b_1, b_2) \in Y$, namely, the points of the form $(P(a,b_1), Q(a,b_2))$ for $a \in A$. In particular, $|E| \geq n(n-N)^2\slash 2d$ when $n$ is larger than $N$.
Hence, we get $\max\{|P(A,B)|,|Q(A,B)|\}\geq \alpha n^{5/4}$ for a constant $\alpha$ depending only on $k$, the complexity of $\phi$, and $N$ which in turn depends only on $\deg P$, $\deg Q$, and $N$.
\end{proof}

  Recall that, by Corollary~\ref{cor:structure}, Theorem~\ref{thm:main}  reduces to the special case where $P(x,y)$ and $Q(x,y)$ are each either additive or multiplicative. We will make use of the following easy lemma which is a restatement of~\cite[Lemma 9]{BasuRaz} using Fact~\ref{fact: 3}. This can be viewed as either a ``definable Ramsey'' theorem.

\begin{fact} \label{fact:101}
Suppose $X \subseteq \RR^m$ is  $\RR$-definable, and $(X_b)_{b\in Y}$ is an $\RR$-definable  family of subsets of $X$ with  $\dim_\RR(X_b)= \dim_\RR (X)$ for every $b\in Y$. Then there are $\RR$-definable $X
' \subseteq X$ and $\RR$-definable $Y'\subseteq Y$ such that  $\dim_\RR (X') =  \dim_\RR (X)$, $\dim_\RR (Y') = \dim_\RR (Y)$,  and  $X'$  is a subset of $X_b$ for all $b \in  Y'$.
\end{fact}

  We identify the underlying set of $\CC$ with $\RR^2$ in the standard way. Then every $\CC$-definable set can be viewed as an $\RR$-definable set. Note that if $X$ is $\CC$-definable and $\dim_\CC (X) =n$, then $\dim_\RR (X) = 2n$. Indeed, using Fact~\ref{fact: 2}(iii), it suffices to consider the case where $X$ is an affine variety. By Noether normalization lemma and Fact~\ref{fact: 3}(iii), we can further reduce it to the case where $X$ is $\CC^n$, and the desired conclusion follows from Fact~\ref{fact: 3}(i-ii).

\meno  Assume  $K$ is either $\RR$ or $\CC$ and $S$ is a subset of $K$. A {\bf decomposition} of $\C_{P,Q}(S)$ consists of a finite set $I$, and an $\RR$-definable family $\C^i_{P,Q}(S)=(C^i_{b_1, b_2})_{(b_1, b_2) \in S^2}$  for each $i \in I$ such that for all $(b_1, b_2) \in S^2$, we have $C_{b_1,b_2}^i\subseteq C_{b_1,b_2}$, and
$$   
\dim_\RR \Big( C_{b_1, b_2}  \setminus \bigcup_{i \in I} C^i_{b_1, b_2}\Big) < \dim_\RR (K).
$$ 
We say that  $\C^i_{P,Q}(S)$ for $i \in I$ as above is {\bf $K$-scattered} if for all $(b_1, b_2) \in S$ 
$$
\dim_\RR \big\{ (b'_1, b'_2) \in S^2 :   \dim_\RR (C^i_{b_1, b_2} \cap C^i_{b'_1, b'_2}) =\dim_\RR (K )\big\} <  \dim_\RR (K).
$$
Note that $\C_{P,Q}(S)$ forms a decomposition of itself. It is easy to see that this notion of being $K$-scattered coincides with the previous one because of the relationship between $\dim_\CC (X)$ and $\dim_\RR (X)$ for a $\CC$-definable set $X$.

\begin{lemma} \label{lem: of prop 1}
Assume $K$ is either $\RR$ or $\CC$, $S$ is a subset of $K$, and $I$ together with $\C^i_{P,Q}(S)$ for $i \in I$ form a decomposition of $\C_{P,Q}(S)$. 
If $\C_{P,Q}(S)$ is not scattered, then for some $i \in I$, $\C^i_{P,Q}(S)$ is not $K$-scattered.
\end{lemma}
\begin{proof}
Recall that $\C_{P,Q}(S)=(C_{b_1, b_2})_{(b_1, b_2) \in S^2}$, and $\C^i_{P,Q}(S)=(C^i_{b_1, b_2})_{(b_1, b_2) \in S^2}$ for $i \in I$.
For $(b_1, b_2) \in S^2$, let $Y_{b_1, b_2} = \{ (b'_1, b'_2) \in S^2 :   \dim_{\RR} (C_{b_1, b_2} \cap C_{b'_1, b'_2}) =\dim_{\RR} (K) \}$ and let 
$$
Y^i_{b_1, b_2} = \big\{ (b'_1, b'_2) \in S^2 :   \dim_{\RR} (C_{b_1, b_2} \cap C^i_{b'_1, b'_2}) = \dim_{\RR} (K)\big\}.
$$ 
Hence, $Y_{b_1, b_2}$ and $Y^i_{b_1, b_2}$ are $\RR$-definable for all $i \in I$ and $(b_1, b_2) \in K^2$ by Fact~\ref{fact: 3}(iii).
Note that 
$\dim_{\RR} (C_{b_1, b_2} \cap C_{b'_1, b'_2}) =\dim_{\RR} (K)$ if and only if $\dim_{K} (C_{b_1, b_2} \cap C_{b'_1, b'_2}) =1$ for all $(b_1, b_2)$ and $(b'_1,b'_2) $ in $S^2$, so $Y_{b_1, b_2}$ is also $K$-definable by Fact~\ref{fact: 3}(iii). 
From the assumption that $\C_{P,Q}(S)$ is not scattered, we obtain $(c_1, c_2) \in K^2$ with $\dim_K(Y_{c_1, c_2}) \geq 1$, or equivalently,  
$$
\dim_\RR(Y_{c_1,c_2}) \geq \dim_\RR (K).
$$
For any $(b'_1, b'_2) \in S^2 $, we have $\dim_{\RR} \big( C_{b'_1, b'_2}  \setminus \bigcup_{i \in I} C^i_{b'_1, b'_2}\big) < \dim_{\RR} K$, so Fact~\ref{fact: 3}(ii) implies there is $i \in I$ such that 
$$
\dim_{\RR} \big(C_{c_1, c_2} \cap C^i_{b'_1, b'_2}\big) \geq  \dim_{\RR} (K) .
$$
It follows that $Y_{c_1, c_2} = \bigcup_{i\in I} Y^i_{c_1, c_2}$. 
Using Fact~\ref{fact: 3}(ii), we get $i \in I$ such that  $\dim_{\RR} (Y^i_{c_1, c_2}) \geq  \dim_{\RR} (K)$. Fix such $i$. Note that $C_{c_1,c_2} \cap C^i_{b'_1,b'_2} \subseteq C_{c_1, c_2} $ and 
$$
\dim_{\RR} \big(C_{c_1, c_2} \cap C^i_{b'_1, b'_2}\big) = \dim_{\RR} (C_{c_1, c_2})  = \dim_{\RR} (K) \quad \text{ for any } (b'_1, b'_2) \in Y^i_{c_1, c_2}.
$$
So we can apply Fact~\ref{fact:101} to get $X' \subseteq K^2$ with $\dim_{\RR} (X') = \dim_\RR (K)$ and $Y' \subseteq K^2$  with $\dim_{\RR} (Y') \geq  \dim_{\RR} (K)$ such that 
$$  X' \subseteq C_{c_1, c_2} \cap C^i_{b'_1, b'_2} \quad \text{for  all } (b'_1, b'_2) \in Y' $$
Hence, we have $\dim (C^i_{b_1, b_2} \cap C^i_{b'_1, b'_2}) \geq \dim_\RR (K)$ for all $(b_1, b_2)$ and $(b_1',b_2')$ in $Y'$. The desired conclusion follows.
\end{proof}

\begin{proposition} \label{prop: special case 1}
If $K$ is either $\RR$ or $\CC$,  $P(x,y)$  and $Q(x,y)$ are individually either additive or multiplicative, and for all cofinite $S \subseteq K$, the family $\C_{P,Q}(S)$ is not scattered, then exactly one of the following possibilities holds:
\begin{enumerate}[{\rm (i)}]
    \item $P(x,y)$ and $Q(x,y)$ form an additive pair over $K$;
    \item $P(x,y)$ and $Q(x,y)$ form a multiplicative pair over $K$.
\end{enumerate}
Moreover, with $N =\deg P + \deg Q+1$, the same conclusion holds with the weaker assumption that $\C_{P,Q}(S)$ is not scattered for all cofinite $S \subseteq K$ with $|K \setminus S|< N$.
\end{proposition}
\begin{proof} Suppose $K$ is either $\RR$ or $\CC$. We first prove the following special case of the proposition.
\begin{claim}\label{clm:x+y and xy}
 Suppose $P(x,y)=u_1(x) + v_1(y)$ and $Q(x,y) = u_2(x)v_2(y)$ where $u_1$, $v_1$, $u_2$, $v_2$ are nonconstant univariate polynomials with coefficients in $K$, and 
 $$ S = \{b\in K : v_1(b) \neq 0, v_2(b)\neq 0\}. $$
Then  $\C_{P,Q}(S)$ is scattered.  \medskip

\noindent\emph{Proof of Claim~\ref{clm:x+y and xy}.} Towards a contradiction, we  assume that $\C_{P,Q}(S)$ is not scattered.  As $u_1(x)$ and $u_2(x)$ are nonconstant, $u_1'(x)$ and $u_2'(x)$ are nonzero. The numerator has higher degree than the denominator in the fraction $\frac{u'_1u_2}{u_2'}(x)$, and so $\frac{u'_1u_2}{u_2'}(x)$ gives us a nonconstant rational function. Hence, we can find a finite $\RR$-definable family  $(U_i)_{i \in I}$ of Euclidean open subsets of $K$ with $$\dim_\RR ( K \setminus \bigcup_{i \in I} U_i )< \dim_\RR (K)$$ and such that for each $i \in I$, $u'_2(x)$ is nonzero on $U_i$, the functions  $u_1(x)$, $u_2(x)$, and $\frac{u'_1u_2}{u_2'}(x)$ induce  $\RR$-definable differentiable homeomorphisms from the $U_i$ to the respective images. For $i \in I$ and $(b_1, b_2) \in K^2$, let $\C^i_{P,Q}(S)$ be the $\RR$-definable family$( C^i_{b_1, b_2}  )_{(b_1, b_2) \in S^2} $ with
$$C^i_{b_1, b_2} = \{  (u_1(a)+ v_1 (b_1), u_2(a)v_2(b_2)) : a \in U_i\}. $$
By Fact~\ref{fact: 3}(iii), the set $I$ together with $\C^i_{P,Q}(S)$ form a decomposition of $\C_{P,Q}(S)$. 
As $\C_{P,Q}(S)$ is not scattered, by Lemma~\ref{lem: of prop 1}, there is $i \in I$ such that $\C^i_{P,Q}(S)$ is not scattered. We obtain, in particular,   $i \in I$ and $(b_1, b_2) \in S^2
$ such that 
$$
\dim_\RR \big\{ (b'_1, b'_2) \in S^2 :   \dim_\RR (C^i_{b_1, b_2} \cap C^i_{b'_1, b'_2}) =\dim_\RR (K )\big\} =  \dim_\RR (K).
$$
Fix $(b'_1, b'_2) \in S^2$ such that
$\dim_\RR (C^i_{b_1, b_2} \cap C^i_{b'_1, b'_2}) = \dim_\RR (K)$, and set 
$$\rho: U_i \to C^i_{b_1, b_2},  a \mapsto (u_1(a)+ v_1(b_1), u_2(a)v_2(b_2))$$
and
$$\rho':  U_i \to C^i_{b'_1, b'_2}, a \mapsto (u_1(a)+ v_1(b'_1), u_2(a)v_2(b'_2)).$$ By our choice of $U_i$, the maps $\rho$ and $\rho'$ are $\RR$-definable  homeomorphisms from $U_i$ onto their respective images. Using Fact~\ref{fact: 3}(vi), choose $C' \subseteq C^i_{b_1, b_2} \cap C^i_{b'_1, b'_2}$ such that $C'$ is an $\RR$-definable open subset of $C^i_{b_1, b_2}$  with  $\dim_\RR (C') = \dim_\RR (K)$.  Set 
$$U' =\rho^{-1}(C') \quad \text{and}\quad  \lambda: U' \to U_i, a \mapsto (\rho')^{-1}(\rho(a)) .$$ Then $U'$ is a definable open subset of $U_i$, the function $\lambda: U' \to U_i $ is $\RR$-definable, analytic, and a homeomorphism onto its image,  and 
$$ (u_1(x)+ v_1(b_1), u_2(x)v_2(b_2)) = (u_1(\lambda(x))+ v_1(b'_1), u_2(\lambda(x))v_2(b'_2)) \text{ on }  U'. $$
Differentiating  $u_1(x)+ v_1(b_1) = u_1(\lambda(x)) +v_1(b'_1) $, we get $ u'_1(x) = u_1'( \lambda(x))\lambda'(x) \text{ on } U'.$
Taking logarithmic derivative of both sides of  $u_2(x)v_2(b_2) = u_2(\lambda(x))v_2(b'_2) $, we get
$$  \frac{u'_2(x)}{ u_2(x)} = \frac{u'_2(\lambda(x))}{ u_2(\lambda(x))} \lambda'(x) \text{ on }  U'.   $$
Dividing the two equations involving $\lambda'(x)$, we get
$$ \frac{u'_1u_2}{u_2'}(x) = \frac{u'_1u_2}{u_2'}(\lambda(x)) \text{ on }  U'. $$
As we have arranged that $\frac{u_1'u_2}{u_2'}$ is a homeomorphism on $U_i$, $\lambda(x) =x$ on $U'$. It follows that $v_1(b_1) = v_1(b'_1)$ and $v_2(b_2) = v_2(b'_2)$. This is a contradiction, as it implies that the set $ \{(b'_1,b'_2)  \in S^2 :   \dim_\RR (C^i_{b_1, b_2} \cap C^i_{b'_1, b'_2}) =\dim_\RR (K) \} $ is finite, which contradicts our earlier choice of $i\in I$ and $(b_1,b_2) \in S^2$.
\end{claim}

The next claim is a strengthening of Claim~\ref{clm:x+y and xy}.

\begin{claim}\label{clm:add and multip}
Suppose $P(x,y) =f( u_1(x) + v_1(y))$ and $Q(x,y) =g( u_2(x)v_2(y))$ where $f, g, u_1, v_1, u_2, v_2$ are univariate polynomials with coefficients in $K$,
 and 
 $$ S = \{b\in K : v_1(b) \neq 0, v_2(b)\neq 0\}. $$
Then $\C_{P,Q}(S)$ is scattered.
\medskip

\noindent\emph{Proof of Claim~\ref{clm:add and multip}.}  
Towards a contradiction, we  assume that $\C_{P,Q}(S)$ is not scattered.
Note that $f, g, u_1, v_1, u_2, v_2$ are nonconstant as $P(x,y)$ and $Q(x,y)$ have nontrivial dependence on $x$ and $y$.
In particular, this allows us to obtain a finite $\RR$-definable family  $(U_i)_{i \in I}$  of Euclidean open subsets of $K$ where  $K \setminus \bigcup_{i \in I} U_i$ is finite such that $f$ and $g$ are injective on $U_i$ for each $i \in I$. Let $\D_{P,Q}(S)=(D_{b_1,b_2})_{(b_1,b_2) \in S^2}$ be the $K$-definable family with
\[
D_{b_1, b_2} = \{  (u_1(a)+ v_1 (b_1), u_2(a)v_2(b_2)) : a \in K\} .
\]
The family $\C_{P,Q}(S) =(C_{b_1, b_2})_{(b_1,b_2) \in S^2}$ satisfies $C_{b_1,b_2} = (f \times g) D_{b_1,b_2}$.
For $(i_1,i_2) \in I^2$ and $(b_1, b_2) \in K^2$, set  
\[
D^{i_1, i_2}_{b_1, b_2} = D_{b_1, b_2}  \cap (U_{i_1}\times U_{i_2})\quad \text{and} \quad  C^{i_1, i_2}_{b_1, b_2}=(f \times g) D^{i_1, i_2}_{b_1, b_2}. 
\]
As $K \setminus \bigcup_{i \in I} U_i$ is finite and  $u_1(x)+ v_1 (b_1)$ and $u_2(x)v_2(b_2) $ are nonconstant polynomials for all $(b_1, b_2)$ is in $S^2$, we get $\dim_\RR( D_{b_1, b_2}  \setminus \bigcup_{(i_1, i_2) \in I^2} D^{i_1, i_2}_{b_1,b_2} ) < \dim_\RR (K)$ for all $(b_1, b_2)$ in $S^2$. It follows that 
\[
\dim_\RR\Big( C_{b_1, b_2}  \setminus \bigcup_{(i_1, i_2) \in I^2} C^{i_1, i_2}_{b_1,b_2} \Big) < \dim_\RR (K) \quad \text{for all } (b_1,b_2) \in S^2.
\]
As $\C_{P,Q}(S)$ is not scattered, we can apply Lemma~\ref{lem: of prop 1} to get $(i_1,i_2) \in I^2$ and $(b_1, b_2) \in S^2$ such that 
\[
\dim_\RR \big\{ (b'_1, b'_2) \in S^2 :   \dim_\RR (C^{i_1, i_2}_{b_1, b_2} \cap C^{i_1,i_2}_{b'_1, b'_2}) =\dim_\RR (K) \big\} \geq \dim_\RR (K).
\]
It is easy to see that $f\times g$ is an  injective and definable map on $U_{i_1} \times U_{i_2}$. By Fact~\ref{fact: 3}(iii),  $\dim_\RR (D^{i_1, i_2}_{b_1, b_2} \cap D^{i_1, i_2}_{b'_1, b'_2}) =\dim_\RR (K)$ whenever $\dim_\RR (C^{i_1, i_2}_{b_1, b_2} \cap C^{i_1, i_2}_{b'_1, b'_2}) =\dim_\RR (K)$ for $(b_1, b_2)$ and $(b'_1, b'_2)$ in $S^2$. Hence, $\D_{P,Q}(S)$ is not scattered, contradicting Claim 1.
\end{claim}

Note that Claim~\ref{clm:add and multip} in particular implies that the two possibilities in the proposition are mutually exclusive. It remains to show under the assumption of the theorem that $P(x,y)$ and $Q(x,y)$ must form an additive pair when they are individually additive, and $P(x,y)$ and $Q(x,y)$ must form a multiplicative  pair when they are individually multiplicative. 

Suppose $P(x,y)$ and $Q(x,y)$ are individually multiplicative. Assume that 
$$P(x,y) = f(u_1(x)v_1(y)) \quad \text{and}  \quad Q(x,y) = g(u_2(x)v_2(y)).$$
where  $f, g, u_1, v_1, u_2, v_2$ are nonconstant univariate polynomials. If $u^m_1(x) = eu^n_2(x)$ and $v^{m'}_1(y) = e'v^{n'}_2(y)$ for $m, n, m', n' \in \NN^{\geq 1}$ and $e, e' \in \CC$, then $P(x,y)$ and $Q(x,y)$ forms a multiplicative pair. So we are left with the case where
$$u^m_1(x) \neq eu^n_2(x)\quad  \text{for all } m, n \in \NN^{\geq 1} \text{ and }e \in \CC.$$
We will in fact show that this cannot happen.
If $\frac{u'_1u_2}{u'_2u_1} =c$ with $c \in \CC$, then by considering behavior when $|x|$ is large, we see that $c$ is a rational number $m \slash n$ and $u_1^m(x) = eu_2^n(x)$  and $e \in \CC$.  Hence, $\frac{u'_1u_2}{u'_2u_1}$ is not a constant function. We then deduce the contradiction 
using the same argument in the proof of Claim~\ref{clm:x+y and xy} and Claim 2 substituting the role of $\frac{u_1'u_2}{u'_2}(x)$ with that of $\frac{u'_1u_2}{u'_2u_1}$.

Similarly, we can rule out the case where $P(x,y)$ and $Q(x,y)$ are each additive but do not form an additive pair using $\frac{u'_1}{u'_2}$ in the place of $\frac{u_1'u_2}{u'_2}(x)$. We are left with the desired possibilities that $P(x,y)$ and $Q(x,y)$ either form an additive pair or a multiplicative pair. Finally, note that with $S$ chosen  in Claim~1 and Claim~2, we have $|K\setminus S|< \deg P+\deg Q+1$, so the same conclusion can be reached under the weaker assumption of the second statement.
\end{proof}

\begin{proof}[Proof of Theorem~\ref{thm:main}]
Let $N =\deg P+ \deg Q+1$ as in Proposition~\ref{prop: special case 1}, and set $\alpha = \alpha( \deg P, \deg Q, N)$  as in Corollary~\ref{cor:st}. So $\alpha$ depends only on $\deg P$ and $\deg Q$, and we can write $\alpha = \alpha(\deg P, \deg Q)$.
Suppose $P(x,y)$ and $Q(x,y)$ do not form a $(5/4, \alpha)$-expanding pair.
Using Corollary~\ref{cor:structure}, we can assume that  $P(x,y)$ and $Q(x,y)$ are each either additive or multiplicative over $K$.
By Corollary~\ref{cor:st}, the family $\C_{P,Q}(S)$ is not scattered for all cofinite $S \subseteq K$ with $|K\setminus S|<N$. Applying Proposition~\ref{prop: special case 1} yields that at least one of the three cases in the statement of the theorem holds, and further the additive and multiplicative cases are mutually exclusive. The conclusion then follows from Proposition~\ref{prop: additivemultiplicativenonexpanding}.
\end{proof}


\begin{remark} \label{Rem: Generalization remark}
The same strategy with obvious modifications will also allow us to prove analogues of Proposition~\ref{prop: special case 1} in more general settings where $P(x,y)$ and $Q(x,y)$ are replaced by definable binary functions in o-minimal expansions of $\RR$ or binary analytic functions restricted to  bounded open sets of $\CC$. Functions of the latter type can be interpreted in an o-minimal expansion of $\RR$, and we expect them to be useful when we consider generalized sum-product phenomena for complex analytic functions, but the corresponding $A$ and $B$ are finite subsets of a fixed bounded open subset of $\CC$. 

As mentioned in the introduction, the analogue of the main theorem should hold in much more general settings. The missing ingredients include appropriate analogue of Szemer\'edi--Trotter Theorem (and Corollary~\ref{cor:st}) and of the Elekes--R\'onyai Theorem in these settings. For o-minimal expansion of $\RR$, the analogue of the Szemer\'edi--Trotter Theorem is known \cite{BasuRaz}. There are not yet strict counterparts of Elekes--R\'onyai results for many o-minimal expansion of $\RR$, but the results in~\cite{ArtemSergeiESomin} are along this direction.
\end{remark}


\begin{remark}\label{Rem: Second comment}
As mentioned earlier in the introduction, a qualitative version of Theorem~\ref{thm:main} without exponent in (i) and without precise form of $P(x,y)$ and $Q(x,y)$ in (ii) and (iii) can be deduced from a result by Bays and Breuillard~\cite[Theorem~1.4]{MartinBays}. They show the following: For a  given algebraic variety $V \subseteq \CC^{k}$ of dimension $d$, if for every $\varepsilon$, there is arbitrarily large $n$ and a grid $A_1 \times \ldots \times A_k \subseteq \CC^k$ with $|A_1| =\ldots =|A_k| =n$ and 
$$|V\cap (A_1 \times \ldots \times  A_k)| > n^{d-\varepsilon},$$
then $(\CC^k, V)$ is in a ``coordinate-wise correspondence'' to a pair $(G, H)$ where the $k$-dimensional group $G$ is a product $G_1^{k_1} \times \ldots \times G_l^{k_l}$ of power of $1$-dimensional algebraic group over the complex numbers, and the $d$-dimensional group $H$ is $H_1 \times \ldots \times H_l$ with $H_i$ a connected algebraic subgroup of $G_i^{k_i}$ for $i \in \{1, \ldots, l\}$.

For given $P(x,y)$ and $Q(x,y)$, set 
$V=\{(a,b,c,d) \subseteq \CC^4 : c= P(a,b), d= Q(a,b)\}.$
Then $V$ has dimension $2$.
Suppose $P(x,y)$ and $Q(x,y)$ do not form a $(1+\varepsilon, \alpha)$ expanding pair for any $\varepsilon, \alpha \in \RR^{>0}$. Then for every $\varepsilon>0$, we can choose arbitrarily large $n$, $A, B \subseteq \CC$ with $|A|=|B|=n$ and $\max{|P(A,B)|, |Q(A,B)|}< n^{1+\varepsilon}$. By the Pigeonhole principle, we can choose $C\subseteq P(A,B)$, and $D \subseteq Q(A,B)$ with $|A|=|B|=|C|=|D|=n$ such that 
$$|V\cap (A\times B \times C \times  D)| > n^{2-2\varepsilon}.$$
Hence, applying the above result and unpacking the notion of ``coordinate-wise correspondence'', we can show that $(\CC^4, V)$ is in a coordinatewise correspondence with $(G^4,H)$ where $G$ is a 1-dimensional abelian algebraic group over the complex numbers, and $H$ a connected subgroup of $G^4$. By standard arguments, one can show that there is open $U \subseteq \CC$  such that on $U$
$$ P(x,y) =  f( u_1(x)+_Gv_1(y)) \text{ and }   Q(x,y) = g(u_2(x)+_Gv_2(y)) $$
with $u_1, v_1, u_2, v_2: U \to G$, $f: (u_1(U)+_G v_1(U)) \to \CC$, and $g: (u_2(U)+_G v_2(U)) \to \CC$  are analytic maps. One can imagine using the assumption that $P(x,y)$ and $Q(x,y)$ are polynomials to show that $G$ must either be the additive or multiplicative groups, and $u_1$, $v_1$, $u_2$, $v_2$, $f$, $g$ must be polynomials of the form in Theorem~\ref{thm: new main}. We do not pursue this further in the current paper.
\end{remark}

\section{Proofs of Theorems~\ref{thm: new main}, \ref{thm:sym}, and other results}

\begin{lemma} \label{lem: uniqueness1}
Suppose $K$ has $\fchar K = 0$, $f$, $\hat{f}$, $u$, $\hat{u}$, $v$, $\hat{v}$ are nonconstant univariate polynomials with coefficients in $K$ such that $f$ and $\hat{f}$ are monic, $u$, $\hat{u}$, $v$, $\hat{v}$ each have constant coefficient equal to $0$, and 
$$  f(u(x)+v(y)) = \hat{f}(\hat{u}(x) + \hat{v}(y)) \quad\text{ for all }x,y\in K.  $$
Then we must have $f=\hat{f}$, $u=\hat{u}$ and $v= \hat{v}$.
\end{lemma}

\begin{proof}
Note that as $\fchar K = 0$, the polynomials $u'$,$v'$,$\hat{u}'$, and $\hat{v}'$ are nonzero.
Taking partial derivatives with respect to $x$ and $y$ and manipulating the equations, we get
$$ \frac{u'(x)}{\hat{u}'(x)} = \frac{v'(y)}{\hat{v}'(y)}.  $$
Hence, both sides of the equation must be equal to a constant $c \in K$. Hence, $u(x)=c\hat{u}(x)+d$ and $v(y)=c\hat{v}(y)+e$. By the assumptions on $u$, $\hat{u}$, $v$, and $\hat{v}$ we see that $d=e=0$. Further, $c=1$ since $f$ and $\hat{f}$ are monic. The desired conclusion follows.
\end{proof}

  The following fact consists of instances of the so-called model-theoretic transfer principle. It is a consequence of the compactness theorem~\cite[Theorem 2.1.4]{Marker}, the downward L\"owenheim--Skolem theorem \cite[Theorem 2.3.7]{Marker}, and the familiar algebraic fact that every countable field of characteristic $0$ is isomorphic to a subfield of $\CC$.

\begin{fact}\label{fact: model trans}
Suppose $\sigma$ is a first-order statement in the language of fields. If $\sigma$ holds for $\CC$, then $\sigma$ holds for all algebraically closed fields with characteristic $0$. If $\sigma$ holds for all subfields of $\CC$, then $\sigma$ holds for all fields with characteristic $0$. If $\sigma$ holds for all fields with characteristic $0$, then there is $N= N(\sigma)$ such that $\sigma$ holds for all fields with characteristic at least $N$.
\end{fact}

\begin{lemma} \label{lem: uniqueness2}
Suppose $K$ has $\fchar K = 0$, $f$, $\hat{f}$, $u$, $\hat{u}$, $v$, $\hat{v}$ are nonconstant univariate polynomials with coefficients in $K$ such that  $u$, $\hat{u}$, $v$, $\hat{v}$ are monic, neither $u(x)v(y)$ nor $\hat{u}(x)\hat{v}(y)$ can be written as $\tilde{u}^n(x)\tilde{v}^n(y)$ where $\tilde{u}$ and $\tilde{v}$ are univariate polynomials over $K$, and $n \geq 2$. If 
$$  f(u(x)v(y)) = \hat{f}(\hat{u}(x)\hat{v}(y)), $$
then we must have $f=\hat{f}$, $u=\hat{u}$ and $v= \hat{v}$.
\end{lemma}
\begin{proof}
By Fact~\ref{fact: model trans}, we can reduce to the case where $K\subseteq \CC$. Taking partial derivatives with respect to $x$ and $y$, and manipulating the equations we get
$$ \frac{u'(x)\hat{u}(x)}{\hat{u}'(x) u(x)} = \frac{v'(y)\hat{v}(y)}{\hat{v}'(y) v(y)}.  $$
Hence, both sides of the equation must be equal to a constant $c \in \CC$. Letting $|x|$ go to infinity, we get $c = \deg{u} \slash \deg \hat{u} = m \slash n$ with $m\slash n$ a rational number in lowest terms. So $n u'(x) \slash u(x) =  m  \hat{u}'(x) /\hat{u}(x)$. Integrating and taking exponential, we get $u^n = d\hat{u}^m$ with $d \in \CC$. By the assumption that $u$ and $\hat{u}$ are monic, $d=1$ and $u^n = \hat{u}^m$. As $K[x]$ is a unique factorization domain, we obtain a univariate polynomial $\tilde{u}$ with coefficients in $K$ such that $ u(x) = \tilde{u}^m(x)$ and $\hat{u}(x) = \tilde{u}^n(x)$. Likewise, we get a univariate polynomial $\tilde{v}$ with coefficients in $K$ such that $ v(y) = \tilde{v}^m(y)$ and $\hat{v}(y) = \tilde{v}^n(y)$. Then $u(x)v(y) = \tilde{u}^m(x) \tilde{v}^m(y)$ and $\hat{u}(x)\hat{v}(y) = \tilde{u}^n(x) \tilde{v}^n(y)$.
By the assumption on $u$, $\hat{u}$, $v$, and $\hat{v}$, we must have $m=n=1$. The desired conclusion follows.
\end{proof}

  We are now ready to deduce Theorems~\ref{thm:sym} and \ref{thm: new main}.
\begin{proof}[Proof of Theorem~\ref{thm:sym}]
Suppose we are not in case (i) of Theorem~\ref{thm:sym}. Then, applying Theorem~\ref{thm:main} on the pair of polynomials $(P(x,y),Q(x,y))$, we find ourselves in case (ii) or (iii) of Theorem~\ref{thm:main}. Suppose we are in case (ii) of Theorem~\ref{thm:main}. Then we get
\[
P(x,y)=f(\gamma_1u(x)+\delta_1v(y)) \quad\text{and}\quad Q(x,y)=g(\gamma_2u(x)+\delta_2v(y)),
\]
where we can further assume that $f$ and $g$ are monic and $u$ and $v$ have constant coefficient 0.
Applying Theorem~\ref{thm:main} again, now on the pair $(\hat{P}(x,y),Q(x,y))$ with $\hat{P}(x,y) = P(y,x)$, we end up in case (ii) of Theorem~\ref{thm:main} again due to mutual exclusivity of the cases of Theorem~\ref{thm:main}. Combined with an application of Lemma~\ref{lem: uniqueness1} on $Q(x,y)$, we get
\[
\hat{P}(x,y)=\hat{f}(\hat{\gamma}_1u(x)+\hat{\delta}_1v(y)) \quad\text{and}\quad Q(x,y)=g(\gamma_2u(x)+\delta_2v(y)),
\]
where $\hat{f}$ is also monic. Hence ${P}(x,y)=\hat{f}(\hat{\delta}_1v(x)+\hat{\gamma}_1u(y))$.
Now we apply Lemma~\ref{lem: uniqueness1} to $P(x,y)$ to conclude the proof for this special case.
We treat the situation where we are in case (iii) of Theorem~\ref{thm:main} similarly, replacing the role of Lemma~\ref{lem: uniqueness1} by that of Lemma~\ref{lem: uniqueness2}. The exactness part of Theorem~\ref{thm:sym} follows easily from exactness part of Theorem~\ref{thm:main}.
\end{proof}

\begin{proof}[Proof of Theorem~\ref{thm: new main}]
It suffices to prove the last assertion as the earlier part is a special case of Theorem~\ref{thm:sym}. The forward direction is immediate, and the backward direction follows from Proposition~\ref{prop: additivemultiplicativenonexpanding}.
\end{proof}

\begin{proposition} \label{prop: transfer}

Suppose $K$ is a field with $\fchar K =0$. We obtain weakened analogues of Theorem~\ref{thm:main}, Theorem~\ref{thm:sym}, and Theorem~\ref{thm: new main}, with ``exactly'' replaced by ``at least'' in the respective statements. Moreover, when $K$ is algebraically closed or real closed, then the full analogues of these results hold.

\end{proposition}

\begin{proof}

We will only prove the proposition for the analogue of Theorem~\ref{thm:main}; the deduction of the analogues of the other two statement from this is similar to what we have done earlier in this section. For a given pair $(d_1, d_2)$ and  $n \in \NN^{\geq 1}$, Theorem~\ref{thm:main} implies the following when $K$ is $\RR$ or $\CC$:   For all tuples $(c_{1, k,l})_{ 0 \leq k+l \leq d_1}$ and $(c_{2,k,l})_{ 0 \leq k+l \leq d_2}$ of elements in $K$, and all subsets $A$ and $B$ of $K$ with $|A| =|B| =n$, with $P(x,y) = \sum_{0 \leq k+l \leq d_1}c_{1, k,l}x^ky^l$ and $Q(x,y) = \sum_{0 \leq k+l \leq d_2}c_{2, k,l}x^{k}y^{l}$
the inequality 
$$\max\{ |P(A,B)|, |Q(A,B)|\} \leq \alpha(\deg P, \deg Q) n^{5/4}$$ implies that $P(x,y)$ and $Q(x,y)$ form either an additive pair or a multiplicative pair. For each such $(d_1, d_2, n)$, the preceding statement admits a first-order expression in the language of fields. Hence, it also holds in all algebraically closed fields of char $0$ and real closed fields, as the respective theories are complete. Therefore, in all algebraically closed fields of char $0$ and real closed fields, the negation of (i) implies (ii) or (iii) in the corresponding analogues of Theorem~\ref{thm:main}. The ``exactness'' part for algebraically closed and real closed fields are simpler and can be achieved similarly.

Now fix a field $K$ with $\fchar K =0$,  and suppose there is $n \in \NN^{\geq 1}$, and finite subsets $A$ and $B$ of $K$ with $|A|=|B|= n$ such that 
$$  \max\{  |P(A,B)|, |Q(A, B)|  \} \geq \alpha(\deg P, \deg Q) n^{ 5/4}.$$
Let $K^\text{a}$ be the algebraic closure of $K$.  Viewing $P(x,y)$ and $Q(x,y)$ as elements of $K^\text{a}[x,y]$ and applying the analogue of 
 Theorem~\ref{thm:main}, we get that either
\[
P=f(\gamma_1u(x)+\delta_1v(y)) \quad\text{ and }\quad Q=g(\gamma_2u(x)+\delta_2 v(y))
\]
or
\[
P=f(u^{m_1}(x)v^{n_1}(y))\quad\text{ and }\quad Q=g(u^{m_2}(x)v^{n_2}(y)),
\]
where $f$, $g$,  $u$, and $v$ are univariate polynomials with coefficients in $K^\text{a}$, $\gamma_1$, $\gamma_2$, $\delta_1$, and $\delta_2$ are in $K^\text{a}$, and $m_1$, $m_2$, $n_1$, and $n_2$ are in $\NN^{\geq 1}$. In the additive case, we can moreover arrange that $f$ and $g$ are monic, and $u$ and $v$ are monic and have zero constant coefficients; in the multiplicative case, we can arrange that $u$ and $v$ are monic, $\gcd(m_1, n_1)=\gcd(m_2,n_2)=1$ and $u(x)v(y)$ can not be written $\tilde{u}^n(x)\tilde{v}^n(y)$ where $\tilde{u}$ and $\tilde{v}$ are univariate polynomials over $K$, and $n \geq 2$. We will show that then $f$, $g$, $u$, and $v$ have coefficients in $K$, and $\gamma_1$, $\gamma_2$, $\delta_1$, $\delta_2$ are in $K$.

Let $G=\mathrm{Aut}(K^\text{a}/K)$ be the absolute Galois group of $K$. It suffices to show that the natural actions of $G$ on $K^{\text{a}}$ and the ring of univariate polynomials with coefficients in $K^{\text{a}}$ fix $f$, $g$, $u$, $v$, $\gamma_1$, $\gamma_2$, $\delta_1$, and $\delta_2$.  We treat $P(x, y)$ in the additive case.  Let $\sigma$ be in $G$. As $P(x,y)$ is in $K[x,y]$, it is fixed by $\sigma$. Hence $$f(\gamma_1u(x)+\delta_1v(y)) = f^\sigma(\gamma_1^\sigma u^\sigma(x)+ \delta_1^\sigma v^\sigma(y))$$ where we use  exponential notation for group actions. Applying  Lemma~\ref{lem: uniqueness1} in $K^\text{a}$, we get $f = f^\sigma$, $u=u^\sigma$, $v=v^\sigma$, $\gamma_1 = \gamma^\sigma_1$, and $\delta_1=\delta_1^\sigma$. We can deal with $Q(x,y)$ identically. The multiplicative case is similar, but  using  Lemma~\ref{lem: uniqueness2} instead of Lemma~\ref{lem: uniqueness1}.
\end{proof}



\begin{remark} \label{rem: ending}
There is evidence that the full analogues of Theorem~\ref{thm:main}, Theorem~\ref{thm:sym}, and Theorem~\ref{thm: new main} do not hold in all fields of characteristic $0$. For example, $x^2+y^2$ is not 1-expanding over $\CC$, but $1$-expanding over $\QQ$ due to a result by Chang~\cite{Chang}.

By the model-theoretic transfer principle (Fact~\ref{fact: model trans}), for each $n$ there is $N$ such that if $A,B \subseteq K$ with  $|A|=|B|=n$, $K$ has positive characteristic at least  $N$ , and  $$\max\{  |P(A,B)|, |Q(A, B)|  \} \leq \alpha(\deg P, \deg Q) n^{ 5/4},$$ then $P(x,y)$ and $Q(x,y)$ must form either an additive  or a multiplicative pair. In more details, for the given $n$, we can obtain a first order statement $\sigma$ indicating that one the three possibilities in Proposition~\ref{prop: transfer} must hold for $n$. Then the statement $\sigma$ holds in all fields of characteristic $0$, so $N$ can be chosen to be $N(\sigma)$ as in Fact~\ref{fact: model trans}.

In the above discussion for fields with positive characteristic, $n$ is generally very small compared to size of the field. There are more involved results due to Tao  along similar lines when $|A|=|B|=n$ is close to the size of the field \cite{Tao}. That suggests that similar results about arbitrary pairs of polynomials should also be true for intermediate values of $n$ in finite fields, but no proof is currently known.
\end{remark}


\section*{Acknowledgments} 
The authors thank Artem Chernikov, Lou van den Dries, Ehud Hrushovski, Brendan Murphy, Sergei Starchenko, Terence Tao, and Frank de Zeeuw  for helpful comments and discussions. Frank was especially generous and outlined for us the proof of certain facts well-known in the field but not explicitly recorded in the literature. In addition,
the authors thank anonymous referees for carefully reading the manuscript and providing many helpful comments.

\bibliographystyle{amsplain}
\bibliography{refs}



\begin{dajauthors}
\begin{authorinfo}[yifan]
  Yifan Jing\\
  Mathematical Institute, University of Oxford\\
  Oxford, UK\\
  yifan\imagedot{}jing\imageat{}maths\imagedot{}ox\imagedot{}ac\imagedot{}uk \\
  \url{https://yifanjing.wordpress.com/}
\end{authorinfo}
\begin{authorinfo}[souktik]
  Souktik Roy\\
  Department of Mathematics, University of Illinois at Urbana-Champaign\\
  Urbana IL, USA\\
  souktik2\imageat{}illinois\imagedot{}edu \\
  \url{https://math.illinois.edu/directory/profile/souktik2}
\end{authorinfo}
\begin{authorinfo}[minh]
  Chieu-Minh Tran\\
  Department of Mathematics, National University of Singapore\\
  Singapore\\
  trancm\imageat{}nus\imagedot{}edu\imagedot{}sg\\
  \url{https://blog.nus.edu.sg/tranchieuminhchieutran/}
\end{authorinfo}
\end{dajauthors}

\end{document}